\documentclass[11pt,dvipsnames]{article}
\usepackage[utf8]{inputenc}
\usepackage[T1]{fontenc}
\usepackage{comment}
\usepackage{authblk}
\usepackage{relsize}
\usepackage{changepage}
\usepackage{mathtools}
\usepackage{booktabs}
\usepackage{multirow}
\usepackage{array}
\usepackage{enumerate}
\usepackage{enumitem}
\usepackage{mathtools}
\usepackage[english]{babel}
\usepackage{amsmath}
\usepackage{mathrsfs}
\usepackage{amsthm}
\usepackage{amsfonts}
\usepackage{url}
\usepackage{amssymb}
\usepackage{graphicx}
\usepackage{tabularx}
\usepackage[numbers]{natbib}
\usepackage{tikz}
\usetikzlibrary{positioning}
\usepackage{color}
\theoremstyle{plain}
\newtheorem{theorem}{Theorem}[section]
 
\newtheorem{defn}[theorem]{Definition}
\newtheorem{proposition}[theorem]{Proposition}
\newtheorem{lemma}[theorem]{Lemma}

\newtheorem{example}[theorem]{Example}

\makeatletter 
\newcommand{\vast}{\bBigg@{4}}
\newcommand{\Vast}{\bBigg@{5}} 
\makeatother 
\definecolor{bulgarianrose}{rgb}{0.28, 0.02, 0.03} 
\definecolor{gray}{rgb}{0.5, 0.5, 0.5} 
 
\theoremstyle{definition}

\newtheorem*{conjecture}{Conjecture}

\theoremstyle{remark} 
\newtheorem*{remark}{Remark} 
\usepackage[left=2cm,right=2cm,top=2cm,bottom=2cm]{geometry} 
\usepackage{amssymb} 
\makeatletter 
\def\namedlabel#1#2{\begingroup
    #2%
    \def\@currentlabel{#2}%
    \phantomsection\label{#1}\endgroup
} 
\usepackage[normalem]{ulem} 
\usepackage{dsfont} 
\usepackage{accents} 
\usepackage{verbatim} 
\usepackage{ulem} 
\usepackage{pgf,tikz,pgfplots} 
\pgfplotsset{compat = 1.16} 
\usepackage[all]{xy} 
 
\usepackage{stackengine} 
\stackMath 
\newcommand\tsup[2][2]{%
 \def\useanchorwidth{T}%
  \ifnum#1>1%
    \stackon[-.5pt]{\tsup[\numexpr#1-1\relax]{#2}}{\scriptscriptstyle\sim}%
  \else%
    \stackon[.5pt]{#2}{\scriptscriptstyle\sim}%
  \fi%
}

\usepackage{enumerate} 
 
\usepackage{amsmath} 
\usepackage{amsfonts} 
\usepackage{amsthm} 
\usepackage{amssymb} 
\usepackage{array} 
\usepackage{booktabs} 
\usepackage{changepage} 
\usepackage[english]{babel} 
\usepackage{enumerate} 
\usepackage{epsfig} 
\usepackage{float} 
\usepackage{framed} 
\usepackage{gensymb} 
\usepackage{graphicx} 
\usepackage[colorlinks=true, allcolors=.]{hyperref}
\usepackage{indentfirst} 
\usepackage{longtable} 
\usepackage{mathtools} 
\usepackage{multirow} 
\usepackage[normalem]{ulem} 
\usepackage{setspace} 
\usepackage{subcaption} 
\usepackage{textcomp} 
\usepackage[utf8]{inputenc} 
\usepackage{verbatim} 
\usepackage{enumerate} 
\usepackage{xlop} 
\usepackage{listings} 
\usepackage[none]{hyphenat}

\newcommand{\sym}{\operatorname{Sym}}
\newcommand{\fsym}{\operatorname{FSym}}
\newcommand{\alt}{\operatorname{Alt}}
\newcommand{\aut}{\operatorname{Aut}}
 
\def\and{%
  \end{tabular}%
  \hskip 1em \@plus.17fil%
  \begin{tabular}[t]{c}}%
  
\title{\scshape Solvability of meromorphic equations in elementary functions}  
 
\author{Miroslav Marinov\footnote{Institute of Mathematics and Informatics, Bulgarian Academy of Sciences, m.marinov1617@gmail.com} \ \ \ Nikola Veselinov\footnote{HSSIMI, Bulgarian Academy of Sciences, nikola.veselinov.veselinov@gmail.com}}
\date{} 
\begin{document}

 \maketitle

\begin{abstract}  
    An equation $f(x)=a$, where $f$ is a complex meromorphic function and $a\in\mathbb{C}$ is a parameter, is solvable in elementary functions if the inverse map $x=f^{-1}(a)$ can be expressed as a finite composition of arithmetic operations (addition, subtraction, multiplication, and division), the exponential function, the complex logarithm, and constants. Specific functions such as $\tan x - x$, $\exp x + x$, $x^x$ have been proven to be unsolvable by Kanel-Belov, Malistov, Zaytsev, while almost all entire surjective functions of at most exponential growth have been covered by Zelenko. All these rely on one-dimensional topological Galois theory, developed by Khovanskii. We generalize to provide a proof for the unsolvability of all elementary meromorphic functions $f$ such that the derivative of $f$ has infinitely many roots $x_i$ and the set of distinct values $f(x_i)$ is infinite.
\end{abstract}

\section{Introduction}

The question of solvability in explicit form, as a finite composition of admissible operations over a base field, is profoundly significant to the development of algebra. After the general solutions by radicals for polynomial equations of degree up to $4$ were discovered, numerous unsuccessful attempts at formulating a closed-form solution for equations of higher degree led to the belief that such solutions simply do not exist. This was proven by Ruffini (with refinements by Cauchy) in $1813$ and also by Abel in $1824$.

\smallskip

Pondering on this question, Abel laid the foundations of the theory of algebraic curves, while Galois was the first to notice that solvability by radicals is dependent on an associated group, now called the Galois group of an algebraic equation. Liouville continued Abel's work on algebraic and differential forms and made significant improvements regarding expressibility by elementary functions, that is, compositions of arithmetic operations, the exponential function $\exp$ and the logarithmic function $\ln$. He rigorously proved the unsolvability in elementary functions of some differential equations \cite{Khovanskii2014}, e.g. $f'(x) = e^{-x^2}$. Later work by Chebyshev, Ritt, Risch, Rosenlicht, and others extended this aspect of Galois theory further.

\smallskip

Arnold discovered topological reasons for the unsolvability of equations -- for instance, see his proof for the quintic \cite{GoldmakherArnoldQuintic}. His student Khovanskii \cite{khovanskii2019dimensionaltopologicalgaloistheory,Khovanskii2014} formalized the methods of the one-dimensional version of topological Galois theory around 1970. The fundamental concepts of topological Galois theory are similar to classical Galois theory, but are suitable for analyzing transcendental functions -- the former one works with extensions of the rational numbers $\mathbb{Q}$, while the latter one is with extensions of the field of rational functions $\mathbb{C}(x)$. The topological criterion for solvability is also based on a group defined through $f$, known as the monodromy group.

\smallskip

Consider the equation $f(x)=a$, where $f:X\to\mathbb{C}$ is a meromorphic elementary function defined on a connected open set $X\subseteq\mathbb{C}$, and $a\in\mathbb{C}$ is a parameter. We aim to determine when the inverse $f^{-1}(a)$ is also an elementary function. All approaches so far rely on analyzing the specific monodromy group and the structure of the roots, hence results up to now have been on individual equations, such as $\tan(x)-x=a$ \cite{Belov_Kanel_2020}, $e^x + x = a$ and $x^x=a$ \cite{kanelbelov2024insolvabilityxxa} by Kanel-Belov, Malistov, Zaytsev.

\smallskip

Other recent developments are by Zelenko \cite{zelenko2021genericmonodromygroupriemann}, proving unsolvability in quadratures, i.e. elementary functions, integration, differentiation and taking the solutions of an algebraic equation, when $f$ is holomorphic, surjective, no two critical points share the same critical value, and for some positive integer $\rho$ \[\lim_{R\to\infty}\sup_{z:|z|\geq R}\frac{\ln|f(z)|}{|z|^\rho}\] is bounded. 

\smallskip

We extend the unsolvability to all elementary meromorphic functions for which the set of critical values
    \[B=\{a\in\mathbb{C}:\ \exists x\in X:f(x)=a,f'(x)=0\}\]
is infinite. 

\begin{theorem}
    Let $X\subseteq\mathbb{C}$ be an open connected set and $f:X\to\mathbb{C}$ be an elementary meromorphic function with infinitely many critical values. Then the equation $f(x)=a$ is unsolvable in elementary functions.
\end{theorem}

The provided proof combines a result of Wielandt on primitive group actions on infinite sets (Theorem \ref{thm:wielandt}), and a relation between primitivity and decompositions (Theorem \ref{prop:imprimitive-equiv-decomposition}) stated, for example, in the doctoral work of Burda on algebraic functions.

\medskip

\section{Auxilliary results}

We give a short review of the required theory, as well as an emphasis on why each condition in the hypothesis of the main theorem matters for our purposes. We follow Khovanskii \cite{Khovanskii2014} for the formulation and development of the aspects of topological Galois theory.

\subsection{Results from Group Theory and Complex Analysis}

For a set $S$ we denote by $\sym(S)$ the symmetric group $S$, by $\alt(S)$ the alternating group on $S$ and by $\fsym(S)$ the subgroup of $\sym(S)$ generated by all cycles of finite length. In all applications here $S$ would be infinite, in particular $\alt(S)$ is not a solvable group.

\begin{defn}
\label{defn:block}
    Let $S$ be a set and let $G$ be a group acting transitively on $S$. A nonempty subset $T\subseteq S$ is called a \emph{block} for $G$ if for each $x\in G$ either $x \cdot T= T$ or $(x \cdot T) \cap T =\emptyset$.
\end{defn}

\begin{remark}
    Every group $G$ acting transitively on $S$ has $S$ and the singletons $\left\{a\ \vert\ a\in S\right\}$ as blocks, which are called \textbf{trivial}. Any other block is \textbf{nontrivial}.
\end{remark}

\begin{defn}
\label{defn:primitive-group}
    Let $S$ be a set and let $G$ be a group acting transitively on $S$. Then $G$ is called \emph{primitive} if it has no nontrivial blocks on $S$.
\end{defn}

\begin{theorem}\emph{(Wielandt, \cite[Theorem 29.2]{neumann2023infinitepermutationgroups})}
\label{thm:wielandt}
    Let $S$ be an infinite set and $G$ be a primitive group acting on $S$. If $G$ contains a non-identity permutation with finite support, then $\alt(S)\leq G$.
\end{theorem}

\begin{remark}
    According to \cite[Theorem 3.3A]{dixon1996permutation}, if $G$ contains a transposition, then even $\fsym(S) \leq G$ holds. There are also versions of Wielandt's theorem when $S$ is finite (e.g., Jordan's Theorem \cite[Theorem 3.3E]{dixon1996permutation}).
\end{remark}

\begin{defn} A function $f:U \to \mathbb{C}$ is \emph{meromorphic} if it satisfies the following conditions:
    \begin{itemize}
        \item It is holomorphic except on a set of isolated points, i.e. such that for any two of these points, there is an open ball containing one, but not the other.
        \item For any of these isolated points $z_0 \in U$, there exists a neighbourhood $V\subseteq U$ of $z_0$ and an integer $n$ such that $(z-z_0)^nf(z)$ is holomorphic and non-zero on $V$.
    \end{itemize}
\end{defn}

\begin{theorem}
\label{theorem:identity-theorem}
    \emph{(Identity theorem)} Let $D\subseteq\mathbb{C}$ be an open set and let $f:D\to\mathbb{C}$ be meromorphic and not identically zero. Then the set of zeroes $\{z\in D:f(z)=0\}$ is a discrete subset of $D$, that is, no accumulation point of zeroes is in $D$.
\end{theorem}

The next lemma confirms that any path can be perturbed in order to not contain a forbidden point, when the forbidden set is finite or countably infinite.

\begin{lemma}
\label{lemma:freeing-path-from-countable-set}
    Let $A$ be an at most countable subset of $\mathbb{C}$, let $\gamma:[0,1]\to\mathbb{C}$ be a continuous path, and $\varphi: [0,1] \to \mathbb{R}_{>0}$ be a continuous function. Then there exists a path $\hat{\gamma}:[0,1]\to\mathbb{C}$ such that we have $\hat{\gamma}(t)\notin A$ and $|\gamma(t)-\hat{\gamma}(t)|<\varphi(t)$.
\end{lemma}

\begin{proof}
    See \cite[Lemma 5.20]{Khovanskii2014}.
\end{proof}

\begin{defn}
\label{defn:critical-point}
    For a meromorphic function $f:U\to\mathbb{C}$, where $U\subseteq\mathbb{C}$, a point $x_0$ is \emph{critical} if $f'(x_0)=0$. In this setting $f(x_0)$ is said to be a \textit{critical value}.
\end{defn}

We emphasize that in the case of a meromorphic function, all critical points have finite order, so the scope of our work does not include functions such as $\exp(1/z)$. In fact, $\exp(1/z) = a$ is solvable via $z = (\ln a)^{-1}$, but this does not contradict the developed theory.

\subsection{Elementary functions}

The elementary functions are compositions of:
\begin{itemize}
    \item arithmetic operations (addition, subtraction, multiplication, and division);
    \item exponential and logarithm functions.
\end{itemize}
In this section we give precise definitions in order to establish their relevant properties in a rigorous manner.

\begin{defn}
\label{defn:differential-field}
    A \emph{differential field} $(K,D)$ is a field $K$ of characteristic zero, equipped with a map $D:K\to K$, satisfying for all $u,v\in K$:
    \begin{enumerate}
        \item $D(u+v) = D(u) + D(v)$;
        \item $D(uv) = uD(v)+vD(u)$.
    \end{enumerate}
\end{defn}

\begin{defn}
\label{defn:exponential-extension}
    Let $(K, D)$ be a differential field. A differential field extension $K\subset L$ is called an \emph{exponential extension} if there exists $y\in L$ satisfying the linear differential equation
    \[D(y)=gy\]
    for some $g\in K$.
\end{defn}

\begin{example}
    {\normalfont We define $\displaystyle \exp(x) = \sum_{i=0}^{\infty} \frac{x^i}{i!}$. Let $K=\mathbb{C}(x)$, the field of rational functions over $\mathbb{C}$. Then the extension $K\subset K(\exp(x))$ is exponential with $g\equiv 1$.}
\end{example}

\begin{defn}
\label{defn:logarithmic-extension}
    Let $(K, D)$ be a differential field. A differential field extension $K\subset L$ is called a \emph{logarithmic extension} if there exists $y\in L$ such that
    \[D(y)\in K,\text{ but }y\notin K\text{.}\]
\end{defn}

\begin{example}
    {\normalfont The natural logarithm is $\ln(x):=\exp^{-1}(x)$, i.e. defined via $y$ such that $\exp(y) = x$, whenever $y$ exists and is unique. Let $K=\mathbb{R}(x)$. Then $K\subset K(\ln(x))$ is a logarithmic extension.}
\end{example}

\begin{defn}
\label{defn:elementary-tower}
    An \emph{elementary tower} over the base differential field $K_0$ is a finite chain
    \[K_0\subset K_1\subset\cdots\subset K_m\text{,}\]
    where each extension $K_{i+1}\supset K_i$ is exponential or logarithmic.
\end{defn}

\begin{defn}
\label{defn:elementary-function}
    A function $f$ is said to be \emph{elementary} if there exists an elementary tower
    \[K_0\subset K_1\subset\cdots\subset K_m\]
    such that $f(x)\in K_m$.
\end{defn}

\begin{example}
    {\normalfont The function $f:\mathbb{C} \to \mathbb{C}$ given by $\displaystyle f(x)=\cos x + \exp(x)-\frac{x}{2}$ is an elementary function due to $\displaystyle \cos x = \frac{\exp(ix) + \exp(-ix)}{2}$ and the tower $K_0 = \mathbb{C}(x)$, $K_1 = K_0(\exp(x))$, $K_2 = K_1(\exp(ix))$.  We emphasize that the $n^\text{th}$ radical $\sqrt[n]{x}$ is not an arithmetic operation but is obtainable by $\exp$ and $\ln$ as \[\sqrt[n]{x}= x^{1/n} = \exp\left(\frac{1}{n}\ln(x)\right)\text{.}\]}
\end{example}

\section{\Large One-dimensional topological Galois theory}
\begin{defn}
\label{defn:solvable-elementary}
    Fix a field $K$, a subset $A \subseteq K$ and a function $f: A \to K$. An equation $f(x)=a$ is \emph{solvable in elementary functions} if there exists an elementary function $g$ such that $f(g(a)) = a$ for all $a$ in the image of $f$.
\end{defn}

\begin{defn}
\label{defn:solvable-group}
    A group $G$ is a \emph{solvable group of depth $\ell\geq 0$} if there exists a nested chain of subgroups $G=G_0\supset\cdots\supset G_\ell=e$, such that for each $i=1,\ldots,\ell$ we have $G_i\triangleleft G_{i-1}$ and the quotient group $G_{i-1}/G_i$ is abelian.
\end{defn}

\begin{defn}
\label{defn:germ}
    Let $X$ be a topological space. Let $f$ and $g$ be functions defined on some neighbourhoods of a point $x_0\in X$. We say that $f$ and $g$ have the same \emph{germ} if there exists a neighbourhood $U$ of $x_0$ such that $f$ and $g$ are defined on $U$ and
    \[f|_U=g|_U\text{.}\]
    The \emph{germ} of $f$ with respect to $x_0$, denoted as $f_{x_0}$, is the set of all functions $g$ satisfying the above relation for some $U$.
\end{defn}

\begin{defn}
\label{defn:forbidden-set}
    An at most countable set $A$ is a \emph{forbidden set} for a germ $f_a$ if $f_a$ admits a regular continuation along every path $\gamma(t)$ with $\gamma(0)=a$ that never intersects the set $A$ except possibly at $a$.
\end{defn}

\begin{defn} \label{defn:singular-point-germ}
    Fix a point $a\in \mathbb{CP}^1$ from the complex projective line. A point $b\in \mathbb{CP}^1$ is \emph{singular} for the germ $f_a$ if there exists a path $\gamma:[0,1]\to \mathbb{CP}^1$, $\gamma(0)=a$, $\gamma(1)=b$ such that the germ has no regular continuation along this path but for every $t$, $0\leq t<1$, it admits a regular continuation along the truncated path $\gamma:[0,t]\to \mathbb{CP}^1$.
\end{defn}

\begin{defn}
\label{defn:equivalent-germs}
    Two germs $f_a$ and $g_b$ defined at points $a$ and $b$ of the complex projective line $\mathbb{CP}^1$ are \emph{equivalent} if the germ $g_b$ is obtained from the germ $f_a$ by analytic continuation along some path. 
\end{defn}

\begin{defn}
\label{defn:s-germ}
    A regular germ is called an $\mathscr{S}$-germ if the set of its singular points is at most countable.
\end{defn}

\begin{defn}
\label{defn:s-function}
    A multivalued analytic function $f:A\to B$, $A,B\subseteq\mathbb{C}$, is a \emph{$\mathscr{S}$-function} if each of its regular germs is an $\mathscr{S}$-germ.
\end{defn}

\begin{theorem}
\label{thm:stability-s-function}
    The class $\mathscr{S}$ of all $\mathscr{S}$-functions is closed under the following operations:
    \begin{enumerate}[label=$($\arabic*$)$]
        \item Differentiation: if $f\in\mathscr{S}$, then $f'\in\mathscr{S}$.
        \item Integration: if $f\in\mathscr{S}$ and $g'=f$, then $g\in\mathscr{S}$.
        \item Composition: if $g,f\in\mathscr{S}$, then $g\circ f\in\mathscr{S}$.
        \item Meromorphic operations: if $f_i\in\mathscr{S}$, $i=1,\ldots,n$, the function $F(x_1,\ldots,x_n)$ is a meromorphic function of $n$ variables, and $f=F(f_1,\ldots,f_n)$, then $f\in\mathscr{S}$.
        \item Solving algebraic equations: if $f_i\in\mathscr{S}$, $i=1,\ldots,n$, and $f^n+f_1f^{n-1}+\ldots+f_n=0$, then $f\in\mathscr{S}$.
        \item Solving linear differential equations: if $f_i\in\mathscr{S}$, $i=1,\ldots,n$, and $f^{(n)}+f_1f^{(n-1)}+\ldots+f_n=0$, then $f\in\mathscr{S}$.
    \end{enumerate}
\end{theorem}

\begin{remark}
    Arithmetic operations and exponentiation are meromorphic operations, hence the class of $\mathscr{S}$-functions is stable under the arithmetic operations and exponentiation.
\end{remark}

\begin{remark}
\label{rem:another-germ}
    Let $f_a$ be the set of all germs of an $\mathscr{S}$-function $f$ at a point $a$ not lying in some forbidden set $A$ $($see Definition~\ref{defn:forbidden-set}$)$. Take a closed path $\gamma$ in $\mathbb{CP}^1\setminus A$ that originates at $a$. Then the continuation of every germ in $f_a$ along $\gamma$ is another germ in $f_a$.
\end{remark}

\begin{theorem}
\label{thm:forbidden-sets-s-germ}
    An at most countable set is a forbidden set of a germ if and only if it contains the set of its singular points. In particular, a germ has a forbidden set if and only if it is a germ of an $\mathscr{S}$-function.
\end{theorem}

\begin{defn}
\label{defn:homomorphism-of-a-monodromy}
    To every path $\gamma$ we assign a map $\tau$ of $f_a$ to itself $($homotopic paths in $\mathbb{C}\setminus A$ give rise to the same $\tau)$. The operation of taking the product of maps, obtained through composition of paths, defines a homomorphism $\rho$ from the fundamental group of $\mathbb{C}\setminus A$ to the group $\aut(f_a)$ of invertible transformations of $f_a$. We call $\rho$ the homomorphism of \emph{$A$-monodromy}.
\end{defn}

\begin{remark}
    By Cayley's theorem every group is isomorphic to a group of permutations, i.e. a subgroup of the symmetric group $S_d$ on $\{1,2,\ldots,d\}$ $($here and throughout, we allow $d=\infty$ and will use it for countable groups$)$. When speaking about monodromy, we will often refer to the corresponding group of permutations rather than the group of automorphisms.
\end{remark}

Key observations about unsolvability are made possible through the study of the \emph{monodromy group} -- the group of all permutations of branches of the $\mathscr{S}$-function $f$ that correspond to loops around the points of the forbidden set $A$. We now proceed to give a precise definition.

\begin{defn}
\label{defn:monodromy-group}
    The \emph{monodromy group} of an $\mathscr{S}$-function $f$ with a forbidden set $A$ is the image of the fundamental group $\pi_1(\mathbb{C}\setminus A,a)$ in the group $\aut(f_a)$ under $\rho$.
\end{defn}

\begin{proposition}
\label{prop:monodromy-properties}
    The following properties hold.
    \begin{enumerate}[label=$($\arabic*$)$]
        \item The $A$-monodromy group of an $\mathscr{S}$-function is independent of the choice of point $a$.
        \item The $A$-monodromy group of an $\mathscr{S}$-function $f$ acts transitively on the branches of $f$.
    \end{enumerate}
\end{proposition}

\begin{proof}
    Both claims follow from Lemma~\ref{lemma:freeing-path-from-countable-set}. We give a proof of the second one.\\
    Let $f_{1,a}$ and $f_{2,a}$ be any two germs of the function $f$ at the point $a$. Since $f_{1,a}$ and $f_{2,a}$ are equivalent, there exists a path $\gamma$ such that under continuation along $\gamma$ the germ $f_{1,a}$ is transformed into $f_{2,a}$. By Lemma~\ref{lemma:freeing-path-from-countable-set} there exists an arbitrarily close path $\hat{\gamma}$ avoiding the set $A$. If $\hat{\gamma}$ is sufficiently close to $\gamma$, then the corresponding permutation of branches takes $f_{1,a}$ to $f_{2,a}$.
\end{proof}

\section{\Large Topological unsolvability in elementary functions}
Consider the equation $f(x)=a$, where $f:X\to\mathbb{C}$ is a meromorphic non-constant function defined on a connected open set $X\subseteq\mathbb{C}$, and $a\in\mathbb{C}$ is a parameter. By Theorem~\ref{theorem:identity-theorem} for any $a$ the set of roots is discrete, hence the inverse $f^{-1}$ is an $\mathscr{S}$-function.

Let $a$ draw a closed curve in the complex plane. Then the roots will also draw curves, as these are preimages of a path-connected set under a continuous map. However, these root curves are not necessarily closed, even though their union has to be a closed set, as it is a preimage of a closed set under a continuous map. We call these curves \emph{local branches}. Recall that a critical point is a root of $f'(x) = 0$.

Denote the local branches by $\{x_1(a_0),\ldots,x_d(a_0)\}$ for some regular value $a_0$, i.e. not a critical point. It is possible that this set is infinite (but countable) and we shall see this option more often in the equations we work with -- write conventionally $d=\infty$ in this scenario. For each $i=1,\ldots,d$ define $x_i^\gamma(a_0)$ as the set of values obtained by analytic continuation of $x_i(a_0)$ along a loop (closed curve) $\gamma$ around a critical point, starting at $a_0$. We require the following crucial condition: 

\begin{quote}
    The curve of $a$ must be such that \emph{any local branch does not contain a critical point}.
\end{quote}

\noindent In other words, the forbidden set from Definition \ref{defn:forbidden-set} is precisely the set of critical points of $f$. Because no root collisions occur along $\gamma$ (a double root must be a root of the derivative, but we assumed $f'\neq 0$ on $\gamma$), each $x_i^\gamma$ overlaps with some $x_j(a_0)$. Thus, $\gamma$ induces a well-defined permutation $\sigma_\gamma\in S_d$, such that
\[x_i^\gamma(a_0)=x_{\sigma_\gamma(i)}(a_0)\text{.}\]
(Intuitively, by using the continuation $\gamma$, we ``move'' from one local branch to another.) The image of the map $\rho:\pi_1(\mathbb{C}\setminus\{\text{critical points}\}, a_0)\to S_d$, given by $\rho(\gamma) = \sigma_\gamma$, is the monodromy group of $f$.

\begin{lemma}
\label{lemma:order-m-loop}
    Suppose $x_c$ is an isolated critical point of a non-constant meromorphic function $f:U\to\mathbb{C}$ and that it has order $m$, that is, $f'(x_c)=f''(x_c)=\ldots=f^{(m)}(x_c)=0$, but $f^{(m+1)}(x_c)\neq 0$. Then for any sufficiently small loop $\gamma$ around the corresponding critical value $a_c:=f(x_c)$, the analytic continuation of the inverse branches of $f$ along $\gamma$ induces a cycle of length $m+1$ on the local branches of the multivalued inverse function $x=f^{-1}(a)$ near $a_c$ and fixes the other preimages.
\end{lemma}

\begin{proof}
    See \cite[p. 304]{Ahlfors1979}. 
\end{proof}

If the region enclosed by $\gamma$ does not contain any critical points, then $\gamma$ is homotopic to a constant map, so none of the roots will be permuted. Therefore it suffices to consider curves which move up to a critical point, make a circle (or a circular arc) around it and then return to the base point to induce a nontrivial permutation.

\begin{lemma}
\label{lemma:induced-cyclic-group-by-elementary-functions}
    The monodromy groups of $\exp$ and $\ln$ are isomorphic to $\mathbb{Z}$.
\end{lemma}

\begin{proof}
    The image of $\exp$ is $\mathbb{C}^\times = \mathbb{C} \setminus \{0\}$. Since by Proposition~\ref{prop:monodromy-properties} monodromy groups are in general independent of the choice of starting point, we may assume $a$ traverses a unit circle $\gamma$ around the origin, starting at $-1$. The local branches are $x_k(a) = \ln a + 2k\pi i$, $k\in \mathbb{Z}$. For $a = e^{i\theta}, \theta \in (-\pi, \pi)$ this becomes $x_k(a) = (2\pi k + \theta)i$. If we analytically continue along $\gamma$, then as $\theta$ runs through $(-\pi, \pi)$, for any $k$ the preimage $x_k(a)$ moves to $x_{k+1}(a)$. This corresponds to the permutation in which $k$ is mapped to $k+1$ for any $k$ and the subgroup generated by it is isomorphic to $\mathbb{Z}$.

    Recall that the values of $\ln z$ for any $z\in \mathbb{C}\setminus \{0\}$ is $\ln |z| + (\arg z + 2\pi n)i$, $n\in \mathbb{Z}$. Let $a$ traverse a unit circle $\gamma$ around the origin, starting at $1$, i.e. $a=e^{i\theta}$, $\theta \in [0, 2\pi)$.  Encircling the origin with winding number $w$, the logarithm is incremented by $2w\pi i$ and the monodromy group is the infinite cyclic group, isomorphic to $\mathbb{Z}$.
\end{proof}

To understand the solvability of the group in the general setting, it will be important only that the groups of $\exp$ and $\ln$ are abelian. For a function call \textit{depth} the number of nested elementary (exponential or logarithmic) extensions needed to obtain a differential field in which the functions is defined, as in Definition \ref{defn:elementary-function}.

\begin{proposition}
\label{prop:solvable-depth-n}
    Let $N$ be a positive integer and $f$ be a function whose inverse is a composition of elementary functions of depth not exceeding $N$. Then the monodromy group of $f$ is solvable.
\end{proposition}

\begin{proof}
    We proceed by induction on $N$. For $N=1$ the function $f^{-1}$ is defined by a single elementary extension over $\mathbb{C}(x)$, so it falls into one of the following:
        \begin{enumerate}[label=(\emph{\alph*})]
            \item Exponential: $f^{-1}=\exp(g(x))$. By Lemma~\ref{lemma:induced-cyclic-group-by-elementary-functions} the monodromy group of $f$ is isomorphic to $\mathbb{Z}$ -- an infinite abelian group, hence solvable.
            \item Logarithmic: $f^{-1}=\ln(g(x))$. By Lemma~\ref{lemma:induced-cyclic-group-by-elementary-functions} the monodromy group of $f$ is isomorphic to $\mathbb{Z}$ -- an infinite abelian group, hence solvable.
        \end{enumerate}

    Assume the result holds for depth $\leq N$. Let $f^{-1}=\phi\circ g$, where $\phi$ is an elementary function of depth $1$ and $g$ is of depth $N$. Then by the inductive hypothesis the monodromy groups $G_g$ and $G_\phi$ of $g^{-1}$ and $\phi^{-1}$, respectively, are solvable. Since $\phi$ is of depth $1$, it represents either an exponential or a logarithmic extension. By Lemma~\ref{lemma:induced-cyclic-group-by-elementary-functions} both $\exp$ and $\ln$ add solvable monodromy, so the monodromy group of $(\phi\circ g)^{-1}$ is also solvable.

    By induction, the inverse of any function defined by a tower of $N$ elementary extensions has a solvable monodromy group.
\end{proof}

\begin{remark}
    The group in the general setting as in Proposition~\ref{prop:solvable-depth-n} may not be abelian -- in some examples one naturally obtains the symmetric or alternating group on the set of positive integers.
\end{remark}

Now we give a property whose proof shows the crucial advantage of the set of roots being discrete.

\begin{lemma}
\label{lemma:complete-group-perm-transitive-on-roots}
    The complete group of permutations acts transitively on roots.
\end{lemma}

\begin{proof}
    If we take a root and move it along a path $\gamma(t)$, connecting it with another root, then $a(t)=f(x(t))$ draws a closed curve, because both $x(0)$ and $x(1)$ are roots for the same $a$. The curve $\gamma$ must avoid points where $f(x)$ is critical, but since the set of these is discrete, such $\gamma$ always exists.
\end{proof}

\begin{theorem}
\label{thm:topological-galois-criterion}
    \emph{(Topological Galois criterion)} The equation $f(x)=a$ is solvable in elementary functions if and only if the monodromy group of $f(x)$ is solvable.
\end{theorem}

\section{Primitivity and main proof}

\begin{defn}
\label{defn:branch-locus}
    Let $f:X\to\mathbb{C}$ be a non-constant meromorphic map from a connected open set $X$. Its \emph{branch locus}
    \[B=\{a\in\mathbb{C}:\ \exists x\in X:f(x)=a,f'(x)=0\}\]
    is the discrete set of critical values.
\end{defn}

When the branch locus is finite, the solvability of $f(x) = a$ reduces to computing a finite monodromy group. The theory on this subject is well developed $($for instance, see \cite{GoldmakherArnoldQuintic,khovanskii2019dimensionaltopologicalgaloistheory,Khovanskii2014}$)$, so for convenience we shall throughout work only on functions with an \textbf{infinite branch locus}. Note that this means that not only the critical points but also the critical values are infinitely many. For example, the function $f(x)=(\exp(x)-1)^2$ has infinitely many critical points of order $1$, but they all have only one corresponding critical value; hence, there is no contradiction between our result on unsolvability and the fact that $(\exp(x) - 1)^2 = a$ is solvable in elementary functions.

\smallskip

Primitivity of the monodromy group is an exceptionally powerful tool when seeking to prove unsolvability. As it turns out, every primitive monodromy group generated by infinitely many cycles (i.e. loops around infinitely many critical values) guarantees unsolvability, thanks to Wielandt's theorem.

\begin{proposition}
\label{prop:wielandt-corollary}
    Let $X \subseteq \mathbb{C}$ be an open connected set and $f:X\to\mathbb{C}$ be a meromorphic map with infinite monodromy group $G$. Suppose $G$ is primitive. Then $G$ contains an unsolvable subgroup, and $f(x)=a$ is unsolvable in elementary functions.
\end{proposition}

\begin{proof}
    Since ramification indices are finite, the induced cycles in $G$ are of finite support; hence, by Theorem~\ref{thm:wielandt}, the group $G$ contains $\alt(\mathbb{Z}_{>0})$, which is unsolvable.
\end{proof}

Recall that in our setting, the function $f(x)$ in the equation $f(x)=a$ is meromorphic and elementary (hence finitely expressed). In the rest of this section, we will leverage this to prove that, even if the monodromy group of $f$ is imprimitive, its (un)solvability is always determined by the monodromy group of a function present in a decomposition of $f$ of the form $(g_1\circ g_2\circ\ldots\circ g_\ell)(x)$.

\smallskip

A result by Ritt \cite{ritt-a609405a-a74b-3dda-b565-a3d841154886} shows that, in the case of polynomials and rational functions, imprimitivity is equivalent to the existence of a nontrivial decomposition of $f$ as $(g\circ h)(x)$. A similar result holds in topological Galois theory.

\begin{proposition} \emph{\cite[Theorem 27]{Burda2012Topological}}
\label{prop:imprimitive-equiv-decomposition}
    Let $X \subseteq \mathbb{C}$ be an open connected set and $f:X\to\mathbb{C}$ be a meromorphic map with monodromy group $G$. Then $G$ is imprimitive if and only if $f$ can be expressed as a composition $(h\circ g)(x)$ of two meromorphic functions $g:X\to Y$ and $h:Y\to\mathbb{C}$ for some open connected set $Y\subseteq\mathbb{C}$. Moreover, the monodromy groups of $g$ and $h$ are quotients of $G$.
\end{proposition}

\begin{proof}
    Denote $f:(X,x_0)\to(Z,z_0)$, where $X$ and $Z$ are topological spaces with some fixed points $x_0 \in  X$ and $z_0 \in Z$, respectively. We will also use this notation with $(Y,y_0)$.

    Firstly, suppose that $(X,x_0)\xrightarrow[]{f}(Z,z_0)$ can be decomposed into $(X,x_0)\xrightarrow[]{g}(Y,y_0)\xrightarrow[]{h}(Z,z_0)$. Then $f^{-1}(z_0)$ can be split into blocks of preimages under $g$ of the points in $h^{-1}(z_0)$. These blocks are always permuted among themselves by any loop in $(Z,z_0)$, thus not all blocks are trivial.

    In the other direction, let $G$ act imprimitively on $f^{-1}(z_0)$. Then there exist blocks $B_0,B_1,\ldots,B_n$, not all trivial, where without loss of generality $x_0 \in B_0$. Denote by $h:(Y,y_0)\to(Z,z_0)$ the map that stabilizes the block $B_0$. Moreover, by the definition of a block, if a loop stabilizes $x_0$, it must also stabilize $B_0$; hence, the induced homomorphism of $f$ is contained in the induced homomorphism of some $g:(X,x_0)\to(Y,y_0)$ and factors through $h$, finishing the proof.
\end{proof}

The action of decomposing a meromorphic function with an imprimitive monodromy group can be repeated on the members of the decomposition until a chain of meromorphic functions with primitive monodromy groups is obtained.

\begin{proposition}
\label{prop:finite-chain-all-primitive}
    Let $X \subseteq \mathbb{C}$ be an open connected set and $f:X\to\mathbb{C}$ be an elementary meromorphic function. Then $f$ can be expressed as a finite chain
    \[f(x)=(g_1\circ g_2\circ\ldots\circ g_\ell)(x),\ \ell\in\mathbb{Z}_{>0}\text{,}\]
    where the monodromy group of each $g_i$, $1\leq i\leq\ell$, is primitive and is also a quotient of the monodromy group of $f$.
\end{proposition}

\begin{proof}
    If the monodromy group of $f$ is primitive, we fix $g_1:=f$ and the proof is complete. Suppose that it is imprimitive. Then by Proposition~\ref{prop:imprimitive-equiv-decomposition} $f$ can be decomposed into $g_1\circ g_2$. Applying the same logic exhaustively to the members of the decomposition, given that $f$ is elementary (i.e. finitely expressed), we are guaranteed to reach a finite chain consisting only of functions with primitive monodromy groups.
\end{proof}

We are now ready to prove our main result.

\begin{theorem}
\label{thm:unsolvability-infinite-branch-locus}
    Let $X \subseteq \mathbb{C}$ be an open connected set and $f:X\to\mathbb{C}$ be an elementary meromorphic function with an infinite branch locus and infinite monodromy group $G$. Then the equation $f(x)=a$ is unsolvable in elementary functions.
\end{theorem}

\begin{proof}
    By Proposition~\ref{prop:finite-chain-all-primitive} we have $f(x)=(g_1\circ g_2\circ\ldots\circ g_\ell)(x)$, where each $g_i$ has a primitive monodromy group and $\ell$ is a positive integer. If every primitive factor $g_i$ had a finite branch locus, then the entire composition would have a finite branch locus (see, for instance, \cite{ritt-a609405a-a74b-3dda-b565-a3d841154886}), which contradicts the assumption that $f$ has an infinite branch locus. Therefore, at least one $g_j$ has an infinite branch locus. Note that by the definition of a meromorphic function, the monodromy group of $g_j$ contains an element with finite support. As a consequence, by Proposition~\ref{prop:wielandt-corollary} the monodromy group of $g_j$ contains $\alt(\mathbb{Z}_{>0})$ and is unsolvable. But by Proposition~\ref{prop:finite-chain-all-primitive}, the monodromy group of $g_j$ is a quotient of the monodromy group of $f$, which leads to unsolvability of $f$ and concludes the proof.
\end{proof}

\medskip 
We conclude with the following hypothesis, based on modest computational evidence.

\begin{conjecture}
\label{conj:any-solvable-tower-core}
    Let $X\subseteq\mathbb{C}$ be a connected open set and $f: X \to \mathbb{C}$ be a meromorphic elementary function. If the equation $f(x)=a$ is solvable in elementary functions, then there exists a finite sequence $g_1,g_2,\ldots,g_\ell$ of solvable functions of depth at most $1$, such that $f(x)=(g_1\circ g_2\circ\ldots\circ g_\ell)(x)\text{.}$
\end{conjecture}

\medskip

\noindent \textbf{Acknowledgments}. This research was supported by the Summer Research School at the High School Student Institute of Mathematics and Informatics (HSSIMI) under an ``Education with Science'' grant in Bulgaria and the provided atmosphere for work is greatly appreciated.

\nocite{*}
\bibliographystyle{plain}
\bibliography{refs_arxiv}

\end{document}